\documentclass[12pt, a4paper]{amsart}
\usepackage{amsmath,amssymb}

\title{A new construction of the real numbers by
alternating series}

\author{Soichi Ikeda}

\address{Graduate School of Mathematics, Nagoya University,
Furocho, Chikusaku, Nagoya 464-8602, Japan}

\email{m10004u@math.nagoya-u.ac.jp}



\theoremstyle{plain}
\newtheorem{theorem}{Theorem}
\newtheorem{lemma}{Lemma}

\newtheorem{definition}{Definition}
\newtheorem{proposition}{Proposition}

\theoremstyle{remark}
\newtheorem{remark}{Remark}

\numberwithin{theorem}{section}
\numberwithin{lemma}{section}
\numberwithin{corollary}{section}
\numberwithin{definition}{section}
\numberwithin{proposition}{section}
\numberwithin{remark}{section}
\numberwithin{equation}{section}

\begin{document}
\begin{abstract}
We put forward a new method of constructing the complete ordered
field of real numbers from the ordered field of
rational numbers. Our method is a generalization of that of
A. Knopfmacher and J. Knopfmacher. Our result implies that there
exist infinitely many ways of constructing the complete ordered
field of real numbers. As an application of our results, we
prove the irrationality of certain numbers.
\end{abstract}

\maketitle

\section{Introduction}
The purpose of this paper is to put forward a new method of
constructing the complete ordered field of real numbers
from the ordered field of rational numbers.
Our method is similar to the method which was put forward
by A. Knopfmacher and J. Knopfmacher in \cite{knopf_alt},
but our method is more general. Moreover our result
gives infinitely many ways of constructing the complete
ordered field of real numbers. As an application of our results,
we prove the irrationality of certain series.

A. Knopfmacher and J. Knopfmacher constructed the complete ordered
field of real numbers by the Sylvester expansion and
the Engel expansion in \cite{knopf_pos} and by the
alternating-Sylvester expansion
and the alternating-Engel expansion in \cite{knopf_alt}.
The advantages of these constructions are the fact that
those are concrete and do not depend on the notion of equivalence classes.
The alternating-Sylvester expansion and the alternating-Engel
expansion are generalizations of Oppenheim's expansion
(see \cite{oppenheim}) and special cases of the alternating
Balkema-Oppenheim's expansion (see \cite{knopf_opp}),
which were introduced by
A. Knopfmacher and J. Knopfmacher in \cite{knopf_alt}.
The definition of the alternating-Sylvester expansion and
the alternating-Engel expansion are the following.

\paragraph{(i) \textbf{Alternating-Sylvester expansion}}
Let $\alpha \in \mathbb{R}$, $a_0 = [\alpha]$ and
$A_1 = \alpha - a_0$. We define, for $n \in \mathbb{N}$ and
$A_n > 0$,
\[ a_n = \Bigl[ \frac{1}{A_n} \Bigr] \]
and
\[ A_{n+1} = \frac{1}{a_n} - A_n. \]
Then
\begin{equation}  \label{def_alt_syl}
  \alpha = a_0 + \frac{1}{a_1} - \frac{1}{a_2} +
    \frac{1}{a_3} - \dots,
\end{equation}
where $a_1 \ge 1$ and $a_{n+1} \ge a_n (a_n +1)$
for $n \in \mathbb{N}$.

\vspace{2ex}

\paragraph{(ii) \textbf{Alternating-Engel expansion}}
Let $\alpha \in \mathbb{R}$, $a_0 = [\alpha]$ and
$A_1 = \alpha - a_0$. We define, for $n \in \mathbb{N}$ and
$A_n > 0$,
\[ a_n = \Bigl[ \frac{1}{A_n} \Bigr] \]
and
\[ A_{n+1} = 1 - a_n A_n. \]
Then
\begin{equation} \label{def_alt_eng}
  \alpha = a_0 + \frac{1}{a_1} - \frac{1}{a_1 a_2} +
     \frac{1}{a_1 a_2 a_3} - \dots,
\end{equation}
where $a_1 \ge 1$ and $a_{n+1} \ge a_i + 1$ for $n \in \mathbb{N}$.

\vspace{2ex}

The relation
\begin{equation} \label{ineq_frac1}
  \frac{1}{d+1} < \alpha \le \frac{1}{d} \qquad
    (\alpha \in (0,1], \, d = [\alpha^{-1}])
\end{equation}
is used in these expansions. We introduce a new series expansion for
every real numbers by using a more general relation
\[ \frac{c}{d+1} < \alpha \le \frac{c}{d}
     \qquad (\alpha \in (0,1], \, c \in \mathbb{N}, \,
     d = [c \alpha^{-1}]). \]

\vspace{2ex}

\begin{definition}[Generalized alternating-Sylvester expansion]
Let $\alpha \in \mathbb{R}$, $q_0 = [\alpha]$ and
$A_1 = \alpha - q_0$.
Let $\{c_n\}_{n=1}^{\infty}$
be a sequence of positive integers.
We define, for $n \in \mathbb{N}$,
\begin{gather}
  a_n = \Bigl[ \frac{c_n}{A_n} \Bigr] \qquad
    (\text{for $A_n \neq 0$}),  \notag  \\
  q_n = \begin{cases}
          \frac{c_n}{a_n} & (A_n \neq 0) \\
          0 & (A_n = 0)
        \end{cases}  \notag
\end{gather}
and
\[ A_{n+1} = q_n - A_n. \]
Then
\begin{equation}  \label{gene_alt_syl}
  \alpha = q_0 + \sum_{n=1}^{\infty} (-1)^{n-1} q_n.
\end{equation}
\end{definition}

\vspace{2ex}

If we regard the alternating-Sylvester series (\ref{def_alt_syl})
as an analogue of the simple continued fraction
\[ a_0 + \cfrac{1}{a_1 +
         \cfrac{1}{a_2 +
         \cfrac{1}{a_3 + \cdots}}}, \]
the generalized alternating-Sylvester series (\ref{gene_alt_syl})
is an analogue of the continued fraction
\[ a_0 + \cfrac{c_1}{a_1 +
         \cfrac{c_2}{a_2 +
         \cfrac{c_3}{a_3 + \cdots}}}. \]
Therefore we can expect that if we take some appropriate $\{c_n\}$,
then we can get a simple series representation for some real
numbers.

The outline of this paper is the following. In Section 2 we
study some fundamental properties of the generalized
alternating-Sylvester series. In Section 3 we take an
arbitrary sequence of positive integer $\{c_n\}_{n=1}^{\infty}$
such that $c_n \mid c_{n+1}$ for all $n \in \mathbb{N}$, and
we prove that the set
\begin{equation}  \label{def_scn}
  S(\{c_n\}) =
  \{ \{q_n\}_{n=0}^{\infty} \mid \text{$\{q_n\}$
  appears in (\ref{gene_alt_syl})} \}
\end{equation}
can be identified with the complete ordered field of real numbers
$\mathbb{R}$ by
introducing the relation $<$ and the operator $+$ and $\cdot$.
In other words we prove that $S(\{c_n\})$ becomes an ordered field
which is isomorphic to $\mathbb{R}$. Since there exist infinitely
many $\{c_n\}$ such that $c_n \mid c_{n+1}$, this implies that
there exist infinitely many ways of constructing the complete
ordered field of real numbers.
Our construction is similar
to that in \cite{knopf_alt}. Therefore our construction is also
concrete and does not use the notion of equivalence classes.
When we prove that
$S(\{c_n\})$ becomes an ordered field, we use a general lemma
(see Lemma \ref{lem_main}). It seems that this lemma can be
used in \cite{knopf_prod}, \cite{knopf_pos} and \cite{knopf_alt}.
In section 4, we prove the irrationality of certain
series by Proposition \ref{prop_ratio} and Proposition
\ref{prop_tcn}.

\begin{remark}
It seems that we can define generalized alternating-Engel
series as follows :

Let $\alpha \in \mathbb{R}$, $A_1 = \alpha - a_0$ with
$0 < A_1 \le 1$, $a_0 \in \mathbb{Z}$. Let $\{c_n\}$ be
a sequence of positive integers. We define,
for $n \in \mathbb{N}$ and $A_n \neq 0$,
\[ a_n = \Bigl[ \frac{c_n}{A_n} \Bigr] \]
and
\[ A_{n+1} = c_n - a_n A_n. \]
Then
\[ \alpha = a_0 + \frac{c_1}{a_1} - \frac{c_2}{a_1 a_2} +
   \frac{c_3}{a_1 a_2 a_3} - \dots. \]

\vspace{2ex}

However, $a_{n+1} \ge a_n$ does not hold in this series.
For example, if we set $A_1 = \alpha = 5/7$, $c_1 = 2$ and
$c_2 = 1$, then $a_1 = 2$, $A_2 = 4/7$ and $a_2 = 1$.
This is a trouble. In order to simplify the argument we
do not argue on this series.
\end{remark}

\section{Fundamental properties of the generalized alternating
Sylvester series}
In this section, we take an arbitrary sequence of positive integers
$\{c_n\}_{n=1}^{\infty}$ and fix it.

\begin{proposition}  \label{prop_fund_propt}
The generalized alternating-Sylvester series has the following
properties for $n \in \mathbb{N}$.

\vspace{2ex}

\begin{enumerate}
\item If $A_n \neq 0$, then we have
\[ \frac{c_n}{a_n + 1} < A_n \le \frac{c_n}{a_n}. \]

\item If $A_{n+1} \neq 0$, then we have
\[ a_{n+1} + 1 > \frac{c_{n+1}}{c_n} a_n (a_n + 1). \]

\item The evaluation $A_n \ge A_{n+1}$ holds.
If $A_n \neq 0$, then we have $A_n > A_{n+1}$.

\item The evaluation $q_n \le 1$ holds.

\item If $A_{n+1} \neq 0$, then we have $a_{n+1} > a_n$.

\item If $A_n \neq 0$, then we have $A_{n+1} < \frac{1}{a_n + 1}$.

\item The evaluation $q_n \ge q_{n+1}$ holds.
If $q_{n+1} \neq 0$, then we have $q_n > q_{n+1}$.
\end{enumerate}
\end{proposition}

\begin{proof}
(1) This trivially follows from the definition of the generalized
alternating-Sylvester espansion.

(2) From (1) and the definition, we have
\begin{equation*}
  \begin{split}
    a_{n+1} + 1 &> \frac{c_{n+1}}{A_{n+1}} \\
                &= \frac{c_{n+1}}{\frac{c_n}{a_n} - A_n} \\
                &> \frac{c_{n+1}}{\frac{c_n}{a_n} -
                  \frac{c_n}{a_n + 1}}  \\
                &= \frac{c_{n+1}}{c_n} a_n (a_n +1).
  \end{split}
\end{equation*}

(3) In the case $A_n = 0$, we have $A_n \ge A_{n+1}$.
For $A_n \neq 0$, we have
\[ A_{n+1} < \frac{c_n}{a_n} - \frac{c_n}{a_n +1}
   \le \frac{c_n}{a_n +1} < A_n. \]

(4) By (3), we have $A_n < 1$ for all $n$. Hence
\[ a_n = \Bigl[ \frac{c_n}{A_n} \Bigr] \ge c_n \]
holds. This implies (4).

(5) From (2), we have (5) by using (4).

(6) By (4), we have
\[ A_{n+1} < \frac{c_n}{a_n} - \frac{c_n}{a_n +1} =
   \frac{c_n}{a_n (a_n + 1)} \le \frac{1}{a_n +1}. \]

(7) In the case $q_{n+1} = 0$, we have $q_n \ge q_{n+1}$.
For $q_{n+1} \neq 0$, we have
\[ q_{n+1} < \frac{c_{n+1}}{c_{n+1} {q_n}^{-1} (a_n +1) -1}
   \le \frac{c_{n+1}}{c_{n+1} a_n + c_{n+1} -1}
   \le \frac{1}{a_n} \le q_n \]
by (2) and (4).
\end{proof}

\begin{remark}  \label{rem_An}
Since we have
\[ \sum_{k=1}^n (-1)^{k-1} q_k =
     A_1 + (-1)^{n-1} A_{n+1}
       \qquad (\text{for all $n \in \mathbb{N}$}), \]
the series in (\ref{gene_alt_syl}) converges by
Proposition \ref{prop_fund_propt}. Hence
\begin{equation}  \label{eq_An}
  (-1)^{n-1} \sum_{k=n}^{\infty} (-1)^{k-1} q_k =
  (-1)^{n-1} \sum_{k=n}^{\infty} (-1)^{k-1} (A_{k+1} + A_k) = A_n
\end{equation}
holds for all $n \in \mathbb{N}$.
\end{remark}

In order to prove Proposition \ref{prop_ord} we require
some lemmas.

We can easily see that the following lemma holds.

\begin{lemma}  \label{lem_int_frac}
Let $c,d \in \mathbb{N}$ and $\alpha \in (0,1]$. Then

\begin{enumerate}
\item there does not exist $d^{\prime} \in \mathbb{Z}$ such
that
\[ \frac{c}{d+1} < \frac{c}{d^{\prime}} < \frac{c}{d}, \]

\item $d = [c \alpha^{-1}]$ is equivalent to
\[ \frac{c}{d+1} < \alpha \le \frac{c}{d}. \]
\end{enumerate}
\end{lemma}

\vspace{1ex}

\begin{lemma}  \label{lem_appro_frac}
Let $\alpha, \alpha^{\prime} \in (0,1]$, $c \in \mathbb{N}$,
$d = [c/ \alpha]$ and $d^{\prime} = [c/ \alpha^{\prime}]$.
If $c/d \neq c/d^{\prime}$ then $\alpha < \alpha^{\prime}$
is equivalent to $c/d < c/d^{\prime}$.
\end{lemma}

\begin{proof}
First, we assume $\alpha < \alpha^{\prime}$. Since
$c/(d+1) < \alpha \le c/d$ and
$c/(d^{\prime} + 1) < \alpha^{\prime} \le c/d^{\prime}$ hold
by Lemma \ref{lem_int_frac} (2), it is sufficient that
we consider the following cases.

\begin{enumerate}
\item $\alpha < \alpha^{\prime} \le c/d$.
\item $c/(d^{\prime} + 1) < \alpha \le c/d < \alpha^{\prime}$.
\item $\alpha \le c/(d^{\prime} + 1) < \alpha^{\prime}$.
\end{enumerate}

If (1) holds, then we have
\[ \frac{c}{d+1} < \alpha < \alpha^{\prime} \le \frac{c}{d}. \]
This implies that $c/d = c/d^{\prime}$ by
Lemma \ref{lem_int_frac} (2), which is impossible.

If (2) holds, then we have
\[ \frac{c}{d^{\prime} + 1} < \frac{c}{d} < \alpha^{\prime}
   \le \frac{c}{d^{\prime}}, \]
which is impossible by Lemma \ref{lem_int_frac} (1).

If (3) holds, then we have
\[ \alpha \le \frac{c}{d} \le \frac{c}{d^{\prime} + 1}
   < \alpha^{\prime} < \frac{c}{d^{\prime}} \]
by Lemma \ref{lem_int_frac} (1).

Next, we assume $c/d < c/d^{\prime}$. Since
$c/(d^{\prime} + 1) < c/d$ is impossible by Lemma
\ref{lem_int_frac} (1), we have
\[ \alpha \le \frac{c}{d} \le \frac{c}{d^{\prime} + 1}
   < \alpha^{\prime}. \]
\end{proof}

\begin{proposition}  \label{prop_ord}
Let $\alpha, \alpha^{\prime} \in \mathbb{R}$ with
$\alpha \neq \alpha^{\prime}$. We define ${a_n}^{\prime}$,
${A_n}^{\prime}$ and $q_n^{\prime}$ as $a_n$, $A_n$ and $q_n$
which appear in the generalized alternating Sylvester expansion of
$\alpha^{\prime}$, respectively. Let
\[ i = \min \{ j \in \mathbb{N} \cup \{0\} \mid q_j \neq q_j^{\prime} \}. \]
Then $\alpha < \alpha^{\prime}$ is equivalent to
\[ \begin{cases}
       q_0 < q_0^{\prime} & (i = 0),  \\
       q_i < q_i^{\prime} & (2 \nmid i),  \\
       q_i > q_i^{\prime} & (\text{$2 \mid i$ and $i \ge 2$}).
     \end{cases} \]
\end{proposition}

\begin{proof}
First, we consider the case $i=0$. If $\alpha < \alpha^{\prime}$, then
we have $q_0 = [\alpha] \le [\alpha^{\prime}] = q_0^{\prime}$.
Therefore we obtain $q_0 < q_0^{\prime}$.
On the other hand, if $q_0 < q_0^{\prime}$,
then we have $[\alpha] < [\alpha^{\prime}]$. Therefore we obtain
$\alpha < \alpha^{\prime}$.

Next, we assume $i \neq 0$. Then we can write
\begin{equation}  \label{rel_series}
  \alpha = q_0 + \sum_{k=1}^{i-1} (-1)^{k-1} q_k
    + (-1)^{i-1} A_i, \qquad
    \alpha^{\prime} = q_0 + \sum_{k=1}^{i-1} (-1)^{k-1} q_k
    + (-1)^{i-1} A_i^{\prime}
\end{equation}
by Remark \ref{rem_An}. These relations imply that
$\alpha < \alpha^{\prime}$ is equivalent to
\[ \begin{cases}
     A_i < A_i^{\prime} & (2 \nmid i),  \\
     -A_i < -A_i^{\prime} & (\text{$2 \mid i$ and $i \ge 2$}).
   \end{cases} \]
By Proposition \ref{prop_fund_propt} (1) and
Lemma \ref{lem_appro_frac}, this is equivalent to
\[ \begin{cases}
     q_i < q_i^{\prime} & (2 \nmid i),  \\
     q_i > q_i^{\prime} & (\text{$2 \mid i$ and $i \ge 2$}).
   \end{cases} \]
This implies the proposition.
\end{proof}

\vspace{2ex}

In order to consider the case $\alpha \in \mathbb{Q}$
we prove the next lemma.

\begin{lemma}  \label{lem_ratio}
Let $c \in \mathbb{N}$ and $p/q \in \mathbb{Q} \cap (0,1]$ with
$p,q \in \mathbb{N}$. Let $d = [cq/p]$. Then the numerator
of $c/d - p/q$ is less than $p$. In other words, $cq - dp < p$.
\end{lemma}

\begin{proof}
We have
\[ cq-dp = cq - \Bigl( \frac{cq}{p} -
     \Bigl\{ \frac{cq}{p} \Bigr\} \Bigr) p \le \frac{p-1}{p} p
     = p - 1, \]
where $\{x\} = x - [x]$.
\end{proof}

\begin{proposition}  \label{prop_ratio}
The real number $\alpha$ is rational if and only if
there exists an $m \in \mathbb{N}$ such that $q_m = 0$.
\end{proposition}

\begin{proof}
If there exists an $m \in \mathbb{N}$ such that $q_m = 0$,
then $\alpha$ is rational.
We assume $\alpha = p/q$, where $p,q \in \mathbb{Z}$ and $q \neq 0$.
Without loss of generality, we may assume that $q_0 = 0$,
$A_1 = p/q$ and $p,q > 0$. By the definition of $a_n, A_n$
and Lemma \ref{lem_ratio}, the numerator of $A_n$ is
strictly monotonically decreasing. This implies the
proposition.
\end{proof}

\vspace{2ex}

Propositions \ref{prop_fund_propt}, \ref{prop_ord} and
\ref{prop_ratio} imply that the generalized alternating-Sylvester
series is similar to alternating-Sylvester series.

\section{construction of the real numbers}
In this section we take an arbitrary sequence of positive
integers $\{c_n\}_{n=1}^{\infty}$ which satisfies the condition
$c_n \mid c_{n+1}$ for all $n \in \mathbb{N}$ and fix it.
Moreover we identify $\{q_n\}_{n=0}^{\infty} \in S(\{c_n\})$
with $(q_0, q_1, q_2, \dots)$.

\begin{remark}  \label{rem_cn}
On the condition $c_n \mid c_{n+1}$ for any $n \in \mathbb{N}$,
the inequality in Proposition \ref{prop_fund_propt} (2)
becomes
\[ a_{n+1} \ge \frac{c_{n+1}}{c_n} a_n (a_n + 1). \]
If the equality holds in the above and $q_{n+2} = 0$, then we have
\[ A_n = q_n - q_{n+1} = \frac{c_n}{a_n + 1}. \]
This contradicts the definition of $q_n$, hence
$q_{n+2} \neq 0$ or
\[ a_{n+1} > \frac{c_{n+1}}{c_n} a_n (a_n + 1) \]
holds.
\end{remark}

In Section 1, we assumed the existence of the real numbers,
and we defined $S(\{c_n\})$ in (\ref{def_scn}). In order to
use $S(\{c_n\})$ for the construction of the real numbers,
here we remove that assumption.

\begin{definition} \label{def_tcn}
Let $\{a_n\}_{n=1}^{\infty}$ be a sequence of positive integers.
We define $\{a_n\} \in U(\{c_n\})$  if and only if
\begin{equation} \label{ineq_an_ucn}
  a_{n+1} \ge \frac{c_{n+1}}{c_n} a_n (a_n + 1)
\end{equation}
holds for all $n \in \mathbb{N}$.

Let $\{q_n\}_{n=0}^{\infty}$ be a sequence of rational numbers.
We define $\{q_n\} \in T(\{c_n\})$ if and only if

\begin{enumerate}
\item $q_0 \in \mathbb{Z}$,

\item $q_n \le 1$ for all $n \in \mathbb{N}$,

\item if $q_1 = 1$, then $q_2 \neq 0$,

\item if $q_m = 0$ for $m \in \mathbb{N}$, then $q_n = 0$ for all $n \ge m$,

\item there exists a $\{a_n\} \in U(\{c_n\})$ such that
$q_n = c_n/a_n$ for all $n \in \mathbb{N}$ if $q_n \neq 0$, and

\item if $q_{n+1} \neq 0$, then $q_{n+2} \neq 0$ or
\[ a_{n+1} > \frac{c_{n+1}}{c_n} a_n (a_n + 1) \]
\end{enumerate}
holds.
\end{definition}

\vspace{2ex}

We can easily see that the following lemma holds.

\begin{lemma}  \label{lem_tcn}
Let $\{q_n\} \in T(\{c_n\})$ and $n \in \mathbb{N}$.

\begin{enumerate}
\item $a_{n+1} > a_n$.
\item $q_{n+1} \le \frac{1}{a_n + 1}$.
\item $q_{n+1} \le q_n$. If $q_{n+1} \neq 0$, then
$q_{n+1} < q_n$.
\item The series
\[ \sum_{k=1}^{\infty} (-1)^{k-1} q_k \]
converges.
\end{enumerate}
\end{lemma}

\begin{proposition}  \label{prop_tcn}
$S(\{c_n\}) = T(\{c_n\})$.
\end{proposition}

\begin{proof}
$S(\{c_n\}) \subset T(\{c_n\})$ trivially follows by Proposition
\ref{prop_fund_propt} and Remark \ref{rem_cn}. In order to
prove $S(\{c_n\}) \supset T(\{c_n\})$, we take
$\{q_n^{\prime}\} \in T(\{c_n\})$ and assume that
$q_0^{\prime} \in \mathbb{Z}$ and $q_n^{\prime} = 0$ or
$q_n^{\prime} = c_n/a_n^{\prime}$ for all $n \in \mathbb{N}$.
Since we can set
\[ \alpha = q_0^{\prime} +
     \sum_{k=1}^{\infty} (-1)^{k-1} q_k^{\prime} \]
by Lemma \ref{lem_tcn} (4), we have
\[ \alpha = q_0 + \sum_{k=1}^{\infty} (-1)^{k-1} q_k \]
by the generalized alternating-Sylvester expansion.
It is sufficient to prove that $q_n = q_n^{\prime}$
for all $n \in \mathbb{N} \cup \{0\}$. Since the case
$q_1^{\prime} = 0$ is trivial, we may assume $q_1^{\prime} \neq 0$.
By considering $[\alpha]$, we have $q_0 = q_0^{\prime}$ and
$\alpha - q_0 = A_1 \le q_1^{\prime} = c_1/a_1^{\prime}$.
If $q_1^{\prime} = A_1$, then $q_1 = q_1^{\prime}$
by Lemma \ref{lem_int_frac} (2).
If $q_1^{\prime} \neq A_1$, then we have
$A_1 \ge q_1^{\prime} - q_2^{\prime} \ge c_1/(a_1^{\prime} + 1)$ by
(\ref{ineq_an_ucn}) and Definition \ref{def_tcn} (5).
However, $A_1 = q_1^{\prime} - q_2^{\prime} =
c_1/(a_1^{\prime} + 1)$ is impossible because of
Definition \ref{def_tcn} (6). Thus we obtain
$c_1/(a_1^{\prime} + 1) < A_1 < c_1/a_1^{\prime}$.
This implies $q_1 = q_1^{\prime}$ by Lemma
\ref{lem_int_frac} (2).

Next we suppose that $q_{n-1} = q_{n-1}^{\prime}$ holds
for $n > 1$. Then we have
\[ (-1)^{n-1} A_n =
   \alpha - q_0 - \sum_{k=1}^{n-1} (-1)^{k-1} q_k =
   \sum_{k=n}^{\infty} (-1)^{k-1} q_k^{\prime} \]
by Remark \ref{rem_An}. Hence we obtain
$A_n \le q_n^{\prime} = c_n/a_n^{\prime}$.
If $q_n^{\prime} = A_n$, then $q_n^{\prime} = q_n$
by Lemma \ref{lem_int_frac} (2).
If $q_n^{\prime} \neq A_n$, then we have
$A_n \ge q_n^{\prime} - q_{n+1}^{\prime} \ge c_n/(a_n^{\prime} + 1)$
by (\ref{ineq_an_ucn}) and Definition \ref{def_tcn} (5).
By Definition \ref{def_tcn} (6) we obtain
$A_n > c_n/(a_n^{\prime} + 1)$.
Since this implies $q_n = q_n^{\prime}$, we obtain
the assertion of the proposition inductively.
\end{proof}

In the rest of this section, we set $S = S(\{c_n\})$ for
simplicity, and we introduce a relation $<$ and operators $+$,
$\cdot$ for $S$.

First we define the binary relation $<$ on $S$.

\begin{definition}  \label{def_ord}
Let $\{p_n\}, \{q_n\} \in S$ with
$\{p_n\} \neq \{q_n\}$ and
\[ i = \min\{ j \in \mathbb{N} \cup \{0\} \mid p_j \neq q_j \}. \]
We define $\{p_n\} < \{q_n\}$ if and only if
\[   \begin{cases}
       p_0 < q_0 & (i = 0),  \\
       p_i < q_i & (2 \nmid i),  \\
       p_i > q_i & (\text{$2 \mid i$ and $i \ge 2$}).
     \end{cases} \]
\end{definition}

\begin{proposition}  \label{prop_ord_cons}
For any $\{p_n\}, \{q_n\}, \{r_n\} \in S$,
we have

\begin{enumerate}
\item $\{p_n\} < \{p_n\}$ does not hold (irreflexive law),
\item $\{p_n\} < \{q_n\}$ or $\{p_n\} = \{q_n\}$ or
$\{q_n\} < \{p_n\}$ (trichotomy),
\item if $\{p_n\} < \{q_n\}$ and $\{q_n\} < \{r_n\}$
then $\{p_n\} < \{r_n\}$ (transitive law).
\end{enumerate}

In other words, $<$ is a linear order in the strict sense on $S$.
\end{proposition}

\begin{proof}
We can easily see that (1) and (2) hold. In order to prove (3), we define
\[ i_1 = \min \{ j \in \mathbb{N} \cup \{0\} \mid p_j \neq q_j \},
   \qquad
   i_2 = \min \{ j \in \mathbb{N} \cup \{0\} \mid q_j \neq r_j \} \]
and $i = \min \{ i_1, i_2 \}$. Then
\[ p_k = q_k = r_k \quad
     \text{(for any $k \in \{0,1, \dots , i-1\})$} \]
and
\[ p_i \neq q_i \quad \text{or} \quad q_i \neq r_i \]
hold. If $i$ is odd, then we have
\[ \begin{cases}
     \text{$p_i < q_i$ and $q_i < r_i$} & (i = i_1 = i_2),  \\
     \text{$p_i = q_i$ and $q_i < r_i$} & (i = i_2 \neq i_1),  \\
     \text{$p_i < q_i$ and $q_i = r_i$} & (i = i_1 \neq i_2).
   \end{cases} \]
Therefore we obtain $p_i < r_i$. The other cases can be proved by
the same argument.
\end{proof}

\vspace{1ex}

If we define
\[ Q_S = \{ \{q_n\} \in S \mid \text{there exists an $m \in \mathbb{N}$
   such that $q_m = 0$} \}, \]
we can identify $Q_S$ with $\mathbb{Q}$ by
Proposition \ref{prop_ord} and \ref{prop_ratio}. In short, the map
\[
   \mathbb{Q} \ni
   \biggl( q_0 + \sum_{n=1}^{\infty} (-1)^{n-1} q_n \biggr)
   \mapsto \{q_n\} \in Q_S \]
is an order-isomorphism. Hence we may regard as $\mathbb{Q} \subset S$.

\begin{theorem}  \label{th_sup_inf}
Let $M$ be a non-empty subset of $S$. If $M$ is bounded from
above (below), then there exists a supremum (an infimum).
\end{theorem}

\begin{proof}
Since $M$ is bounded from above, there exists a $d_0$ such that
\[ d_0 = \max \{ q_0 \in \mathbb{Z} \mid
     \text{there exists a $(q_0,q_1, \dots) \in M$} \}. \]
If there does not exist an upper bound for $M$ such that
$(d_0, q_1, \dots) \in S$, then $(d_0 + 1, 0, \dots)$ is
a supremum for $M$. We assume that there exists an upper bound
for $M$ such that $(d_0, q_1, \dots) \in S$.
Since there exists a $(q_0,q_1,\dots) \in M$ such that
$q_0 = d_0$, we can define
\[ d_1 = \max \{ q_1 \in \mathbb{Q} \mid
     \text{there exists a $(d_0,q_1,\dots) \in M$} \} \]
from the definition of $S$ and $<$. By the same argument,
we can define
\[ d_2 = \min \{ q_2 \in \mathbb{Q} \mid
     \text{there exists a $(d_0,d_1,q_2,\dots) \in M$} \}. \]
In general, if we have defined $d_{k-1}$ for $k > 1$, then
we define
\[ d_k =
     \begin{cases}
       \max \{ q_k \in \mathbb{Q} \mid
         \exists (d_0,d_1, \dots ,d_{k-1}, q_k, \dots) \in M \} &
         (\text{$k-1$ is even}),  \\
       \min \{ q_k \in \mathbb{Q} \mid
         \exists (d_0,d_1, \dots ,d_{k-1}, q_k, \dots) \in M \} &
         (\text{$k-1$ is odd}).
     \end{cases} \]
By the definition of $<$ and $\{d_n\}$, $\{d_n\}$
is the supremum for $M$. We can prove this as follows.
If $\{d_n\}$ is not an upper bound for $M$, then there
exists a $\{q_n\} \in M$ such that $\{d_n\} < \{q_n\}$.
By setting $i = \min \{ n \in \mathbb{N} \mid d_n \neq q_n \}$,
we have $d_i < q_i$ for odd $i$ or $d_i > q_i$ for even $i$.
This contradicts the definition of $\{d_n\}$. On the other hand,
if $\{d_n\}$ is not minimum upper bound for $M$, then there
exists an upper bound for $M$ $\{r_n\}$ such that $\{r_n\} < \{d_n\}$.
We set $j= \min \{ n \in \mathbb{N} \mid d_n \neq r_n \}$.
By the definition of $\{d_n\}$, there exists an
$X = (x_0, x_1, \dots ) \in M$ such that $x_k = d_k$ for
$0 \le k \le j$. Then we have $\{r_n\} < X \le \{d_n\}$.
This is impossible.

The case of the infimum can be proved by the same argument.
\end{proof}

In order to introduce the algebraic structure for $S$,
we require some preparations.

\begin{definition}  \label{def_lim}
Let $\{a_n\}_{n=1}^{\infty}$ be a sequence of rational
numbers. We define $L(a_n)$ if and only if, for all $m \in \mathbb{N}$,
there exists an $N \in \mathbb{N}$ such that
$|a_n| < 1/m$ holds for all $n \ge N$.
\end{definition}

Note that in the usual sense $L(a_n)$ means
$\displaystyle \lim_{n \to \infty} a_n = 0$.

The following definition and lemma are the same as in \cite{knopf_alt}.

\begin{definition}  \label{def_xn}
Let $X \in S$ with $X = (x_0, x_1, \dots)$.
We define
\[ X_n = (x_0, x_1, \dots , x_n, 0, \dots), \]
where $n \in \mathbb{N}$.
\end{definition}

We can easily see that the next lemma holds.

\begin{lemma}  \label{lem_xn}
Let $X \in S$ with $X = (x_0, x_1, \dots)$.
Then we have

\begin{enumerate}
\item $X_{2n} \le X_{2n+2} \le X \le X_{2n+1} \le X_{2n-1}$,
\item $L(X_{2n-1} - X_{2n})$,
\item $\sup X_{2n} = \inf X_{2n-1} = X$.
\end{enumerate}
\end{lemma}

In order to prove Lemma \ref{lem_main}, we also require the next lemma.

\begin{lemma}  \label{lem_prep}
Let $\{a_n\}$ be a monotonically increasing sequence of
rational numbers which is bounded from above. Let
$X = \sup a_n$. Then we have $L(X_{2n-1} - a_n)$.
\end{lemma}

\begin{proof}
By contradiction. Assume that there exists an $m$ such that
\[ \forall N \in \mathbb{N}, \exists n \in \mathbb{N}
   [\text{$n \ge N$ and $|X_{2n-1} - a_n| = X_{2n-1} - a_n \ge 1/m$}] \]
holds. Since we have $X_{2n-1} - a_n \ge X_{2n+1} - a_{n+1}$
by the assumption of the lemma, we have $X_{2n-1} - a_n \ge 1/m$ for all
$n \in \mathbb{N}$. On the other hand, by Lemma \ref{lem_xn},
there exists an $N \in \mathbb{N}$ such that
\[ X_{2n-1} - X_{2n} < 1/2m \]
holds for all $n \ge N$. Hence we have
\[ \begin{split}
     a_n &\le X_{2n-1} - \frac{1}{m}  \\
         &\le X_{2N-1} - \frac{1}{m}  \\
         &< X_{2N} - \frac{1}{2m}
   \end{split} \]
for $n \ge N$. This implies that $X_{2N} - (1/2)m$ is
an upper bound for $\{a_n\}$. Therefore we obtain
\[ \sup a_n \le X_{2N} - \frac{1}{2m} < X_{2N}
   \le \sup X_{2n} = X. \]
This contradicts the definition of $X$.
\end{proof}

The following lemma is important in the proofs
of algebraic properties of $S$. It seems
that this lemma can be used in the work of
A. Knopfmacher and J. Knopfmacher
\cite{knopf_prod}, \cite{knopf_pos}, \cite{knopf_alt}.

\begin{lemma}  \label{lem_main}
Let $\{a_n\}, \{b_n\}$ be monotonically increasing sequence
of rational numbers which are bounded from above. Then
$\sup a_n = \sup b_n$ is equivalent to $L(a_n - b_n)$.
\end{lemma}

\begin{proof}
First we assume $\sup a_n = \sup b_n$. We set $X = \sup a_n = \sup b_n$.
Since
\[ |a_n - b_n| \le |a_n - X_{2n-1}| + |X_{2n-1} - b_n|, \]
we have $L(a_n - b_n)$ by Lemma \ref{lem_prep}.

Next we assume $L(a_n - b_n)$. By contradiction. Assume
that $\sup a_n \neq \sup b_n$. Without loss of generality,
we may assume $\sup a_n < \sup b_n$. We set $X=\sup a_n$. Then
there exists an $N \in \mathbb{N}$ such that $X_{2n-1} < b_n$
holds for all $n \ge N$.
Since $b_n - X_{2n-1} \le b_{n+1} - X_{2n+1}$ for $n \ge N$,
we have
\[ |b_n - a_n| = (b_n - X_{2n-1}) + (X_{2n-1} - a_n)
   \ge b_N - X_{2N-1} > 0 \]
for $n \ge N$. This contradicts $L(a_n - b_n)$.
\end{proof}

Now we define the operators on $S$, and prove that $S$
is an ordered field. (These definitions are the same as
in \cite{knopf_alt}.)

\begin{definition}
Let $X,Y \in S$. We define the following symbol and
operators.

\begin{enumerate}
\item $0 = (0,0, \dots) (= 0 \in \mathbb{Q})$.
\item $X + Y = \sup(X_{2n} + Y_{2n})$.
\item $-X = \sup(-X_{2n-1})$.
\end{enumerate}
\end{definition}

\begin{definition}
Let $X,Y \in S$. We define the following symbol and
operators.

\begin{enumerate}
\item $1 = (1,0, \dots) (= 1 \in \mathbb{Q})$.
\item 
\[ X \cdot Y =
     \begin{cases}
       \sup(X_{2n} \cdot Y_{2n}) & (X,Y \ge 0), \\
       (-X) \cdot (-Y) & (X,Y \le 0), \\
       -((-X) \cdot Y) & (X \le 0, Y \ge 0), \\
       -(X \cdot (-Y)) & (X \ge 0, Y \le 0).
     \end{cases} \]
\item
\[ X^{-1} =
     \begin{cases}
       \sup((X_{2n-1})^{-1}) & (X > 0), \\
       -((-X)^{-1}) & (X < 0).
     \end{cases} \]
\end{enumerate}
\end{definition}

Since $X_{2n} + Y_{2n} \le X_1 + Y_1$, $-X_{2n-1} \le -X_2$,
$X_{2n} \cdot Y_{2n} \le X_1 \cdot Y_1 \quad (X,Y \ge 0)$ and
$(X_{2n-1})^{-1} \le X_2^{-1} \quad (X > 0)$,
these definitions are possible.

Now we prove that $+$ (resp. $\cdot$) shares the same properties
with the usual addition (resp. multiplication).

\begin{proposition}  \label{prop_add}
Let $X, Y, Z \in S$. We have

\begin{enumerate}
\item $X+Y=Y+X$,
\item $X+0=X$,
\item $(X+Y)+Z=X+(Y+Z)$,
\item $X+(-X)=0$,
\item if $X<Y$, then $X+Z < Y+Z$.
\end{enumerate}
\end{proposition}

\begin{proof}
(1), (2) These trivially follow from the definition of $+$.

(3) We set $A=X+Y$, which means $L(A_{2n} - (X_{2n} + Y_{2n}))$
by Lemma \ref{lem_main}. Since
\[ |(A_{2n} + Z_{2n}) - (X_{2n} + Y_{2n} + Z_{2n})| =
   |A_{2n} - (X_{2n} + Y_{2n})|, \]
we have $L((A_{2n} + Z_{2n}) - (X_{2n} + Y_{2n} + Z_{2n}))$.
By Lemma \ref{lem_main}, this implies
$\sup (A_{2n} + Z_{2n}) = \sup (X_{2n} + Y_{2n} + Z_{2n})$,
hence we obtain
$(X+Y)+Z = \sup (X_{2n} + Y_{2n} + Z_{2n})$.
By the same argument, we can also prove that
$X+(Y+Z)= \sup (X_{2n} + Y_{2n} + Z_{2n})$.

(4) We set $A=-X$, which means $L(A_{2n} + X_{2n-1})$
by Lemma \ref{lem_main}. Since
\[ |X_{2n} + A_{2n}| \le |X_{2n} - X_{2n-1}| + |X_{2n-1} + A_{2n}|, \]
we have $L((X_{2n} + A_{2n}) - 0)$ from Lemma \ref{lem_xn}.
This implies $\sup (X_{2n} + A_{2n}) = \sup 0$ by Lemma \ref{lem_main},
hence we obtain (4).

(5) Since $X_{2n} + Z_{2n} < Y_{2n} + Z_{2n}$ holds for
sufficiently large $n$, we have $X+Z \le Y+Z$. If
$X+Z=Y+Z$, then we have $L((X_{2n} + Z_{2n}) - (Y_{2n} + Z_{2n}))$
by Lemma \ref{lem_main}.
In short we have $L(X_{2n} -Y_{2n})$. However, this is impossible
by Lemma \ref{lem_main}.
\end{proof}

From Proposition \ref{prop_add} (1), (2), (3) and (4),
it follows that $S$ is an abelian group
on $+$, hence we can use $-(-X) = X$, $-(X+Y) = (-X) + (-Y)$,
etc. Moreover we obtain $0<X \Leftrightarrow 0+(-X) < X+(-X)
\Leftrightarrow -X < 0$, $X < 0 \Leftrightarrow 0 < -X$ and
$X < Y \Leftrightarrow -X > -Y$
by Proposition \ref{prop_add} (5).

\begin{proposition}  \label{prop_mul}
Let $X, Y, Z \in S$. We have

\begin{enumerate}
\item $X \cdot Y = Y \cdot X$,
\item $X \cdot 1 = X$,
\item $X \cdot Y = -((-X) \cdot Y) = -(X \cdot (-Y))$,
\item $X \cdot X^{-1} = 1$ $\quad (X \neq 0)$,
\item $(X \cdot Y) \cdot Z = X \cdot (Y \cdot Z)$,
\item if $X < Y$ and $Z > 0$, then $XZ < YZ$.
\end{enumerate}
\end{proposition}

\begin{proof}
(1), (2) These trivially follow from the definition of $\cdot$.

(3) In the case $X, Y \ge 0$, by setting $Z = -X$ and $W = -Y$,
we have
\begin{gather}
-((-X) \cdot Y) = -(Z \cdot Y) = -(-((-Z) \cdot Y)) = X \cdot Y, \notag \\
-(X \cdot (-Y)) = -(X \cdot W) = -(-(X \cdot (-W))) = X \cdot Y. \notag
\end{gather}
From this case, we can prove the other cases. For example, in the
case $X \le 0$, $Y \ge 0$, we have
\[ -(X \cdot (-Y)) = -((-X) \cdot (-(-Y))) = -((-X) \cdot Y) = X \cdot Y.\]

(4)  For $X > 0$, we set $A=X^{-1}$, which means
$L(A_{2n} - (X_{2n-1})^{-1})$ by Lemma \ref{lem_main}.
Since
\[ |X_{2n} A_{2n} - X_{2n} (X_{2n-1})^{-1}| \le
   |X_1| \cdot |A_{2n} - (X_{2n-1})^{-1}|, \]
we obtain $L(X_{2n} A_{2n} - X_{2n} (X_{2n-1})^{-1})$.
This implies $X \cdot X^{-1} = \sup (X_{2n} (X_{2n-1})^{-1})$
by Lemma \ref{lem_main}. On the other hand, since
\[ |X_{2n} (X_{2n-1})^{-1} -1| =
   |(X_{2n-1})^{-1}| \cdot |X_{2n} - X_{2n-1}| \le
   |X_2^{-1}| \cdot |X_{2n} - X_{2n-1}|, \]
we obtain $L(X_{2n} (X_{2n-1})^{-1} -1)$. This implies
$X \cdot X^{-1} = 1$. In the case $X < 0$, by (3), we have
\[ X \cdot X^{-1} = X \cdot (-((-X)^{-1})) = (-X) \cdot (-X)^{-1} = 1. \]

(5) For $X, Y, Z \ge 0$, we can prove (5) by the same
argument as in the proof of Proposition \ref{prop_add} (3). By using
(3), we can prove the other cases from this case.
For example, in the case $X, Z \ge 0$ and $Y \le 0$,
we have
\begin{equation*}
\begin{split}
  (X \cdot Y) \cdot Z &= (-(X \cdot (-Y))) \cdot Z  \\
                      &= -((X \cdot (-Y)) \cdot Z)  \\
                      &= -(X \cdot ((-Y) \cdot Z))
                         \qquad (-Y > 0)  \\
                      &= X \cdot (-((-Y) \cdot Z))  \\
                      &= X \cdot (Y \cdot Z).
\end{split}
\end{equation*}

(6) For $X,Y \ge 0$, we can prove (6) by the same
argument as in the proof of Proposition \ref{prop_add} (5).
From this case, we can also prove the other cases easily.
For example, in the case $X<Y \le 0$, by (3), we have
\[ -(X \cdot Z) = (-X) \cdot Z > (-Y) \cdot Z = -(Y \cdot Z). \]
This implies $X \cdot Z < Y \cdot Z$.
\end{proof}

\begin{proposition}  \label{prop_add_mul}
Let $X, Y, Z \in S$. We have $X \cdot (Y + Z) = X \cdot Y + X \cdot Z$.
\end{proposition}

\begin{proof}
First we assume $X,Y,Z \ge 0$. Let $A = Y + Z$,
$B = X \cdot Y$ and $C = X \cdot Z$. Then
$L(A_{2n} - (Y_{2n} + Z_{2n}))$, $L(B_{2n} - X_{2n} Y_{2n})$
and $L(C_{2n} - X_{2n} Z_{2n})$ holds by Lemma \ref{lem_main}.
Since we have
\[
\begin{split}
  |X_{2n} A_{2n}& - (B_{2n} + C_{2n})|  \\
  &= |X_{2n}(A_{2n} - (Y_{2n} + Z_{2n})) +
    (X_{2n} Y_{2n} - B_{2n}) + (X_{2n} Z_{2n} - C_{2n})|  \\
  & \le |X_1| \cdot |A_{2n} - (Y_{2n} + Z_{2n})| +
    |X_{2n} Y_{2n} - B_{2n}| + |X_{2n} Z_{2n} - C_{2n}|,
\end{split}
\]
we obtain $L(X_{2n} A_{2n} - (B_{2n} + C_{2n}))$. This implies
$\sup(X_{2n} A_{2n}) = \sup (B_{2n} + C_{2n})$ from
Lemma \ref{lem_main}. This implies
$X \cdot (Y + Z) = X \cdot Y + X \cdot Z$.

Next we consider the case $X \ge 0$ and $Y + Z \ge 0$.
Since $Y \ge 0$ or $Z \ge 0$ holds by Proposition \ref{prop_add} (5),
we may assume $Z \le 0$.
Since $-Z \ge 0$, we obtain
\[ X \cdot (Y + Z) + X \cdot (-Z) = X \cdot (Y + Z + (- Z))
   = X \cdot Y. \]
This is equivalent to $X \cdot (Y + Z) = X \cdot Y + X \cdot Z$
by Proposition \ref{prop_mul} (3).

By Proposition \ref{prop_mul} (3), we can easily prove the other cases
from these cases. For example, in the case $X \ge 0$ and $Y+Z \le 0$,
we have
\[ X \cdot (Y+Z) = -(X \cdot ((-Y) + (-Z)))
   = -(X \cdot (-Y) + X \cdot (-Z)) = X \cdot Y + X \cdot Z. \]
\end{proof}

By Propositions \ref{prop_ord_cons}, \ref{prop_add},
\ref{prop_mul} and \ref{prop_add_mul}, $S$ is an ordered
field. Since any ordered field which satisfies Theorem
\ref{th_sup_inf} is isomorphic to $\mathbb{R}$ (see \cite{struct_real}),
we obtain the following theorem.

\begin{theorem}
The set $S$ can be identified with the complete ordered field of real numbers.
\end{theorem}

\section{An application}
Let $\{a_n\}_{n=1}^{\infty}$ be a sequence of positive integers.
For $K \ge 1$, we define $\{a_n\} \in P(K)$ if and only if
$a_{n+1} \ge K a_n (a_n + 1)$ holds for all sufficiently large
$n \in \mathbb{N}$.  For each $\{a_n\} \in P(K)$ we define
\[ f(z; \{a_n\}) = \sum_{n=1}^{\infty} \frac{z^n}{a_n}, \]
which is an entire function.

The purpose of this section is to prove the following theorem
by using some properties of the generalized alternating-Sylvester
expansion.

\begin{theorem}
Let $\{p_n\} \in P(K)$ and $l \in \{1, 2, 3, \dots, [K]\}$.
Then
\[ f(-l; \{p_n\}) = \sum_{n=1}^{\infty} \frac{(-1)^n}{p_n} l^n \]
is an irrational number.
\end{theorem}

\begin{proof}
We assume that $p_{n+1} \ge K p_n (p_n + 1)$ for all $n \ge N$
($N \in \mathbb{N}$). We define $a_n = p_{n + 2N -1}$ and
$c_n = l^{n + 2N -1}$. Then we have
\begin{align*}
  f(-l; \{p_n\}) &= \sum_{n=1}^{2N-1} \frac{(-1)^n}{p_n} l^n
    + \sum_{n=2N}^{\infty} \frac{(-1)^n}{p_n} l^n  \\
  &= \sum_{n=1}^{2N-1} \frac{(-1)^n}{p_n} l^n +
    \sum_{n=1}^{\infty} \frac{(-1)^n c_n}{a_n} = A_1 + A_2,
\end{align*}
say. Note that $A_1 \in \mathbb{Q}$. Since we have
\begin{align*}
  a_{n+1} &= p_{n+2N} \ge K p_{n+2N-1} (p_{n+2N-1} + 1)  \\
    &\ge [K] a_n (a_n + 1) \ge \frac{c_{n+1}}{c_n} a_n (a_n + 1),
\end{align*}
by Definition \ref{def_tcn} and Proposition \ref{prop_tcn},
the series $A_2$ is the generalized alternating-Sylvester
expansion of the number $A_2$. By Proposition \ref{prop_ratio}
we obtain the theorem.
\end{proof}

\end{document}